\documentclass[11pt,reqno,tbtags]{amsart}
\usepackage[]{amsmath,amssymb,amsfonts,latexsym,amsthm,enumerate}
\usepackage{bbm}
\usepackage{enumerate}
\usepackage[numeric,initials,nobysame]{amsrefs}

\numberwithin{equation}{section}

\newtheorem{theorem}{Theorem}[section]
\newtheorem*{theorem*}{Theorem}
\newtheorem{lemma}[theorem]{Lemma}
\newtheorem{claim}[theorem]{Claim}
\newtheorem{proposition}[theorem]{Proposition}

\theoremstyle{definition}{

\newtheorem{definition}[theorem]{Definition}
\newtheorem*{definition*}{Definition}

}

\theoremstyle{remark}{
\newtheorem*{remark*}{Remark}

}

\newcommand{\R}{\mathbb R}

\newcommand{\Z}{\mathbb Z}

\newcommand{\E}{\mathbb{E}}
\renewcommand{\P}{\mathbb{P}}
\DeclareMathOperator{\var}{Var}
\DeclareMathOperator{\Cov}{Cov}

\renewcommand{\epsilon}{\varepsilon}
\renewcommand{\phi}{\varphi}

\newcommand{\ty}{\tilde{Y}}
\newcommand{\tnu}{\tilde{\nu}}
\newcommand{\hpi}{\hat{\Pi}}
\newcommand{\Q}{\mathbb{Q}}

\date{}

\begin{document}
\title{Rigidity and tolerance for perturbed lattices}

\author{Yuval Peres}
\address{Yuval Peres\hfill\break
Microsoft Research\\
One Microsoft Way\\
Redmond, WA 98052-6399, USA.}
\email{peres@microsoft.com}
\urladdr{}

\author{Allan Sly}
\address{Allan Sly\hfill\break
University of California, Berkeley and Australian National University \hfill\break
Department of Statistics\\
367 Evans Hall\\
Berkeley, CA 94720, USA.}
\email{sly@stat.berkeley.edu}
\urladdr{}

\begin{abstract}

A perturbed lattice is a point process $\Pi=\{x+Y_x:x\in \mathbb{Z}^d\}$ where the lattice points in $\mathbb{Z}^d$ are perturbed by i.i.d.\ random variables $\{Y_x\}_{x\in \mathbb{Z}^d}$.  A random point process $\Pi$ is said to be rigid if $|\Pi\cap B_0(1)|$, the number of points in a ball, can be exactly determined given $\Pi \setminus B_0(1)$, the points outside the ball.  The process $\Pi$ is called deletion tolerant if removing one point of $\Pi$ yields a process with distribution indistinguishable from that of $\Pi$. Suppose that  $Y_x\sim N_d(0,\sigma^2 I)$ are Gaussian vectors with with $d$ independent components of variance $\sigma^2$. Holroyd and Soo showed that in dimensions $d=1,2$ the resulting Gaussian perturbed lattice $\Pi$ is rigid and deletion intolerant.  We show that in dimension $d\geq 3$ there exists a critical parameter $\sigma_r(d)$ such that   $\Pi$ is rigid if $\sigma<\sigma_r$ and deletion tolerant (hence non-rigid)  if $\sigma>\sigma_r$.

\end{abstract}

\maketitle

\section{Introduction}\label{sec:intro}

Let $\Pi=\{x+Y_x:x\in \mathbb{Z}^d\}$ denote the lattice   $\mathbbm{Z}^d$ perturbed by independent and identically distributed random variables $\{Y_x\}_{x\in \mathbb{Z}^d}$ taking values in $\R^d$. In this paper we address the questions of rigidity and deletion tolerance of such point processes. Rigidity holds if given the points of $\Pi$ outside a ball, one can determine exactly the   number of points of $\Pi$ inside that ball.

Deletion tolerance concerns the effect of removing a single point.  If one point, say $Y_0$, is removed from $\Pi$, can this be detected?  More formally, are the laws of $\Pi$ and $\Pi_u = \big\{x+Y_x:x\in \mathbb{Z}^d \setminus \{u\}\big\}$ mutually singular for $u \in \mathbb{Z}^d$?

\begin{definition}
A $\Pi$-point is an $\mathbb{R}^d$ valued random variable $Z$ such that $Z\in\Pi$ a.s.  A point process $\Pi$ is {\bf deletion tolerant} if for any $\Pi$-point $Z$, the point process $\Pi \setminus Z$ is absolutely continuous\footnote{When discussing  absolute continuity or singularity of two random objects, we are referring to their laws.} with respect to  $\Pi$. The point process $\Pi$ is {\bf deletion singular} if  $\Pi$ and $\Pi\setminus Z$ are mutually singular for any $\Pi$-point $Z$.  We say that $\Pi$ is {\bf insertion tolerant} if for any  Borel set $V \subset \mathbb{R}^d$ with Lebesgue measure $\mathcal{L}(V)\in(0,\infty)$, if $U$ is independent of $\Pi$ and uniform in $V$ then~$\Pi \cup U$ is absolutely continuous with respect to  $\Pi$.  If $\Pi$ and $\Pi \cup U$ are mutually singular for all such $V$, then we say that $\Pi$ is {\bf insertion singular}.

For a  point process  $\Pi$ and a ball $B\subset\mathbb{R}^d$ we define $\Pi_{\hbox{in}}=\Pi_{\hbox{in}}(B):=\Pi \cap B$ and $\Pi_{\hbox{out}}(B)=\Pi \cap B^c$.  We say that $\Pi$ is {\bf rigid} if for all balls $B\subset\mathbb{R}^d$ there exists a measurable function $N=N_B$ on the collection of discrete point sets in $\mathbb{R}^d$ such that $N_B(\Pi_{\hbox{out}}(B))= | \Pi_{\hbox{in}}(B)| \ a.s.$.
\end{definition}

Rigidity turns out to be closely related to deletion tolerance where we consider removing multiple points.  We write $\Pi_S := \big\{x+Y_x:x\in \mathbb{Z}^d \setminus S \big\}$.

\begin{proposition}\label{p:rigid-k-tolerant}
If the distribution of $Y_x$ has a density which is everywhere positive, then the perturbed lattice $\Pi=\{x + Y_x: x\in\mathbb{Z}^d\}$ is rigid if and only if $\Pi$ and $\Pi_S$ are mutually singular for all finite sets $S\subset \mathbb{Z}^d$.
\end{proposition}

It was shown in~\cite{HolSoo:10} that the perturbed lattice is deletion singular in dimension $d=1$ when the perturbations $Y_x$ have bounded first moment and in dimension $d=2$ when the perturbations have bounded second moment.  In contrast, we show that when $d\geq 3$, the question of deletion tolerance depends more delicately on the law of the perturbations; in particular, for Gaussian perturbations it exhibits a phase transition.

\begin{theorem}\label{t:gaussian}
Let $\Pi$ be the perturbed lattice in $\mathbb{Z}^d$ with   Gaussian $N_d(0,\sigma^2 I)$ perturbations.  For  $d \geq 3$ there exist critical variances $0 < \sigma_r(d) \leq \sigma_c(d)$ such that
\begin{itemize}
  \item If $\sigma>\sigma_c$ then $\Pi$ is deletion tolerant and is mutually absolutely continuous with respect to $\Pi_0$.
  \item If $0<\sigma<\sigma_c$ then $\Pi$ is deletion singular.
  \item If $0<\sigma<\sigma_r$ then $\Pi$ is rigid.
  \item If $\sigma>\sigma_r$ then $\Pi$ is non-rigid.
\end{itemize}
\end{theorem}

We conjecture that in fact $\sigma_c=\sigma_r$ and that for all i.i.d.\ perturbations, the perturbed lattice is rigid if and only if the perturbed lattice is deletion singular.  However, in Theorem~\ref{t:twoButNotOne} we show that for similar point processes these notions may differ.

Given the results of~\cite{HolSoo:10}, it is natural to ask if heavy tailed random variables with infinite means may be deletion tolerant.  In the case of $\alpha$-stable perturbations we give a complete characterization.

\begin{theorem}\label{t:stable}
Let $\Pi$ be a one dimensional perturbed lattice with symmetric $\alpha$-stable perturbations.  If $\alpha<1$ then the perturbed lattice $\Pi$ is deletion and insertion tolerant and mutually absolutely continuous with $\Pi_0$, while if $\alpha \geq 1$ then it is deletion singular and rigid.
\end{theorem}

In Section~\ref{s:generalPerturbation} we give a more general categorization of which perturbations give rise to deletion tolerance and rigidity.

\subsection{Absolute Continuity}
Assuming that the distribution of the perturbations has a density which is everywhere positive, we establish equivalences between the different notions of deletion and insertion tolerance.
\begin{proposition}\label{p:equivalence}
If the distribution of $Y_x$ has a density which is everywhere positive then the following are equivalent
\begin{enumerate}
\item The perturbed lattice is deletion tolerant.\label{it:delTol}
\item The perturbed lattice is insertion tolerant.\label{it:insTol}
\item The perturbed lattice is not deletion singular. \label{it:delSing}
\item The perturbed lattice is not insertion singular. \label{it:insSing}
\item The measures $\Pi$ and $\Pi_0$ are mutually absolutely continuous.\label{it:mutAbs}
\item The measures $\Pi$ and $\Pi_0$ are not mutually singular.\label{it:delSingOld}
\end{enumerate}
\end{proposition}
We will also consider the case where $k$ points are inserted or deleted.   Generalizing the earlier definitions, a point process $\Pi$ is {\bf $k$-deletion tolerant} if for any distinct $\Pi$ points $Z_1,\ldots,Z_k \in\Pi$, $\Pi \setminus Z$ is absolutely continuous with respect to  $\Pi$ and {\bf $k$-deletion singular} if they are always mutually singular.  We say that $\Pi$ is {\bf $k$-insertion tolerant} if for any  Borel set $V \subset \mathbb{R}^d$ with Lebesgue measure $\mathcal{L}(V)\in(0,\infty)$, if $U_1,\ldots,U_k$ are independent points uniform in $V$ and independent of $\Pi$ then $\Pi \cup \{U_1,\ldots,U_k\}$ is absolutely continuous with respect to  $\Pi$.  If $\Pi$ and $\Pi \cup \{U_1, \ldots,U_k\}$ are mutually singular then we say $\Pi$ is {\bf $k$-insertion singular}.

Perhaps surprisingly, there exists a translation-invariant point process $\hat{\Pi}$ that is  deletion singular but not rigid. In fact, such a process can be $2$-deletion tolerant; that is, removing a single point from $\hat{\Pi}$ yields a singular measure, but removing any two points from  $\hat{\Pi}$ yields a process which is absolutely continuous to the original!

We construct $\hat{\Pi}$ as the  union of two correlated perturbed lattices. For  $d\geq 3$ and $0<\delta<\sigma$, let $Y_x$ be i.i.d.\  $N_d\Bigl(0,(\sigma^2-\delta^2)I\Bigr)$ variables. For $(x,i)\in \mathbb{Z}^d\times \{1,2\}$ let $Y'_{x,i}$ be i.i.d. $N_d(0,\delta^2I)$ variables.  Setting $\hat{Y}_{x,i}=Y_x+Y'_{x,i}$, we define the point process $\hat{\Pi}=\{x+\hat{Y}_{x,i}:(x,i)\in \mathbb{Z}^d\times \{1,2\}\}$.  The next theorem is proved in Section~\ref{s:twoButNotOne}.
\begin{theorem}\label{t:twoButNotOne}
 There exist $0<\delta<\sigma$ such that $\hat{\Pi}$ is deletion singular but $2$-deletion tolerant and hence non-rigid.
\end{theorem}

\subsection{Exponential Intersection Tails property}
We say that a measure $\eta$ on oriented paths in the lattice has {\bf Exponential Intersection Tails} with parameter $0< \theta <1$, denoted $EIT(\theta)$, if for some $C>0$,
\[
\eta\times\eta\Big\{(\gamma,\gamma'):|\gamma\cap\gamma'|\geq n \Big\} \leq C\theta^n.
\]
The uniform measure on oriented paths on $\mathbb{Z}^d$ has Exponential Intersection Tails for some $\theta<1$ when $d\geq 4$ but not when $d=3$.  Moreover, in~\cite{BPP:98} a measure on oriented paths in $\mathbb{Z}^3$ was constructed with Exponential Intersection Tails while it was shown that no such measure exists when $d\leq 2$.

\section{Proof of Theorem~\ref{t:gaussian}}

In this section we establish Theorem~\ref{t:gaussian} by first proving results about more general perturbations.
Let $e_1,\ldots,e_d$ denote the standard basis vectors in $\mathbb{R}^d$.

\begin{proposition}\label{p:walkTolerant}
For $d\geq 3$ there exists $\rho(d)>1$ such that the following holds.  Let $\Pi$ be a $d$-dimensional perturbed lattice with i.i.d.\ perturbations $\{Y_x\}_{x\in\mathbb{Z}^d}$ with density $g(x)$ which is everywhere positive.  If
\begin{equation}\label{e:muIcondition}
\max_{i} \int_{\mathbbm{R}^d} \left(\frac{g(x+e_i)}{g(x)} \right)^2 g(x) d x < \rho(d) \, ,
\end{equation}
then the perturbed lattice $\Pi$ is deletion  tolerant.
\end{proposition}


\begin{proof}
 Let $\eta$ denote a distribution over oriented paths with $EIT(\theta)$ for some $\theta=\theta(d)\in(0,1)$. Choose $\rho(d)$ so that $\rho(d)\theta(d)<1$.

By the Cauchy-Schwartz inequality, the hypothesis of the theorem implies that
\begin{equation}\label{e:ratioL2Bound}
\max_{i,j} \int_{\mathbbm{R}^d} \left(\frac{g(x+e_i)}{g(x)} \right)\left(\frac{g(x+e_j)}{g(x)} \right) g(x) dx < \rho(d).
\end{equation}
  Let $\gamma=\{\gamma_0,\gamma_1,\ldots \}\subset\mathbb{Z}^d$ be an oriented walk on $\mathbb{Z}^d$ from the origin, (i.e. $\gamma_0=0$ and $\gamma_i-\gamma_{i-1}$ is a standard basis vector).  We define $Y^\gamma=\{Y^\gamma_x\}_{x\in\mathbb{Z}^d}$ to be a field of independent random variables distributed as
\[
Y^\gamma_x \stackrel{d}{=} \begin{cases}
Y_x + \gamma_{i+1}-\gamma_i &x\in\gamma, x=\gamma_i,\\
Y_x & x\not\in \gamma.
\end{cases}
\]
By construction the point $\gamma_i + Y^\gamma_{\gamma_i}$ has the same distribution as $\gamma_{i+1} + Y_{\gamma_{i+1}}$ so changing the perturbations in this way has the effect of shifting the points on $\gamma$ one step along the path.  As there is then a point centered at every vertex in $\mathbb{Z}^d$ except 0, it follows that the point process $\{x+Y^\gamma_x:x\in \mathbb{Z}^d\}$ has the same law as $\Pi_0$.

Denote by $\nu$ and $\nu^\gamma$  the distributions of $\ty$ and $Y^\gamma$, respectively. We would be done if $Y$ and $Y^\gamma$ were mutually absolutely continuous, but of course they are singular, since we have altered significantly the distribution of a specific infinite sequence of points.  However, using a similar argument to that of~\cite{ACHZ:08} and~\cite{BerPer:13}, this singularity can be smoothed away by averaging over $\gamma$.  Let $\Gamma$ denote a random path with the law $\eta$
satisfying $EIT(\theta)$, and define
$ \ty= Y^\Gamma $.
Thus if $\tnu$  denotes the distribution of $\ty$, then
\[
\tnu=\int \nu^\gamma \ \eta(d\gamma) \, .
\]
We will now show that when $\sigma$ is sufficiently large, the distributions of $\nu$ and $\tnu$ are mutually absolutely continuous.  Denote $Y(m)=\{Y_x\}_{|x|\leq m}$ and let $\nu$ denote the measure induced by $Y$ and $\nu_m$ the measure induced by $Y(m)$.  Define $\ty(m)$, $\tnu$, $\tnu_m$, $Y^\gamma(m)$, $\nu^\gamma$ and $\nu^\gamma_m$ similarly.  We let $L(y)$ denote the Radon-Nikodym derivative $\frac{d \tnu}{d\nu}$ and let $L_m(y)$ denote $\frac{d \tnu_m}{d\nu_m}$.  Observe that $L_m(Y_m)$ is a martingale which converges to $L(Y)$ almost surely.  If $L_m(Y_m)$ is an $L^2$ bounded martingale then $\tnu \ll \nu$ (see, e.g., \cite{MP:10}, Theorem 12.32 or ~\cite{Dur:10}.)  By definition of $L_m$,
\begin{align}\label{e:LM2expression}
\E [L_m(Y_m)]^2 &= \int \left[L_m(y)\right]^2 \nu_m(dy)\nonumber\\
&= \int \left[\int \frac{d\nu^\gamma_m}{d\nu_m}(y) \eta_m(d\gamma) \right]^2 \nu_m(dy)\nonumber\\
&= \int \int \int \frac{d\nu^\gamma_m}{d\nu_m} (y) \frac{d\nu_m^{\gamma'}}{d\nu_m} (y) \eta_m(d\gamma) \eta_m(d\gamma') \nu_m(dy)\nonumber\\
&= \int \int \int \frac{d\nu^\gamma_m}{d\nu_m} (y) \frac{d\nu_m^{\gamma'}}{d\nu_m} (y) \nu_m(dy) \eta_m(d\gamma) \eta_m(d\gamma') .
\end{align}
For fixed $\gamma$ the measure $\nu^\gamma$ is a product measure, so
\begin{equation}\label{e:LMproduct}
\int \frac{d\nu_m^\gamma}{d\nu_m} (y) \frac{d\nu_m^{\gamma'}}{d\nu_m} (y) \nu_m(dy)=\prod_{x:\in\mathbb{Z}^d:|x|\leq m} \int_{\mathbb{R}^d}
\frac{d\nu_{m,x}^\gamma}{d\nu_{m,x}} (y_x) \frac{d\nu_{m,x}^{\gamma'}}{d\nu_{m,x}} (y_x) \nu_{m,x}(dy_x)\, ,
\end{equation}
where $\nu_{m,x}=\mu$ is the distribution of $Y_x$ which is simply $\mu$.  If $x\not\in\gamma$, then $\frac{d\nu_{m,x}^\gamma}{d\nu_{m,x}}=1$ and hence
\begin{align*}
\int_{\mathbb{R}^d}
\frac{d\nu_{m,x}^\gamma}{d\nu_{m,x}} (y_x) \frac{d\nu_{m,x}^{\gamma'}}{d\nu_{m,x}} (y_x) \nu_{m,x}(dy_x)
&= \int_{\mathbb{R}^d} \frac{d\nu_{m,x}^{\gamma'}}{d\nu_{m,x}} (y_x) \nu_{m,x}(dy_x)\\
&=\int_{\mathbb{R}^d} \nu^{\gamma'}_{m,x}(dy_x) =1 \,.
\end{align*}
 A similar result holds when $x\not\in\gamma'$, so it remains to consider $x\in\gamma\cap\gamma'$.  In this case $x=\gamma_{|x|}=\gamma'_{|x|}$ and for some $1\leq j,j'\leq d$, we have  $e_j = \gamma_{|x|+1}-\gamma_{|x|}$ and $e_{j'} = \gamma'_{|x|+1}-\gamma'_{|x|}$.  Then by definition of $\tnu_{m,x}$ and equation~\eqref{e:ratioL2Bound} we have that
\begin{align*}
\int_{\mathbb{R}^d}
\frac{d\nu_{m,x}^\gamma}{d\nu_{m,x}} (y_x) \frac{d\nu_{m,x}^{\gamma'}}{d\nu_{m,x}} (y_x) \nu_{m,x}(dy_x)
&= \int_{\mathbb{R}^d}
\left(\tfrac{g(x+e_j)}{g(x)} \right)\left(\tfrac{g(x+e_{j'})}{g(x)} \right) g(x) dx\leq \rho.
\end{align*}
Defining $N=N_{\gamma,\gamma'}=|\gamma \cap \gamma'|$ and substituting in equation~\eqref{e:LMproduct} we have that
\[
\int \frac{d\nu_m^\gamma}{d\nu_m} (y) \frac{d\nu_m^{\gamma'}}{d\nu_m} (y) \nu_m(dy) \leq \rho^{N}
\]
and so by equation~\eqref{e:LM2expression}
\[
\sup_m \E [L_m(Y_m)]^2 \leq \E \rho^{N} < \infty \, ,
\]
by the $EIT(\theta)$ assumption on $\eta$.
It follows that $L_m(Y_m)$ converges to $L(Y)$ almost surely which is finite $\nu$-almost everywhere and hence that $\tnu$ is absolutely continuous with respect to $\nu$. Since $\tnu$ generates the point process $\Pi_0$ it follows that $\Pi_0$ is absolutely continuous with respect to $\Pi$.  The result then follows by Proposition~\ref{p:equivalence}.
\end{proof}

\subsection{Deletion intolerance for small perturbations}

In this section we show that if the perturbations are small enough then we have deletion intolerance.
Let $\gamma=(u_0,u_1,\ldots)$ denote a nearest neighbor path in $\Z^d$ with $u_0=0$ and let $\gamma_n = (u_0,\ldots,u_n)$.  For an i.i.d. field $\{Y_u\}_{u\in \Z^d}$ let $M_{n,d}$ denote
\[
M_{n,d} := \sup_{\gamma} \sum_{u\in \gamma_n} Y_u.
\]
and
\[
M_d := \limsup \frac1n M_{n,d}.
\]
Since $M_d$ is not affected by changing a finite number of $Y_u$ it is almost surely constant depending only on the distribution of $Y_u$ so we will denote this constant as $M_d(Y)$. A simple union bound over paths implies that $M_d$ is finite when $Y$ is Gaussian while Theorem~1 of~\cite{GanKes:94} implies that $M_d(Y)$ is finite provided that
\begin{equation}\label{e:greedyCondition}
\E Y^d \log^{d+\epsilon} Y < \infty.
\end{equation}
We have the following result when $M_d(|Y|_1)<\frac12$.

\begin{lemma}\label{l:greedyIntolerant}
Suppose that $Y_x$ has an absolutely continuous distribution with respect to Lebesgue measure on $\mathbb{R}^d$, the $\ell^1$ norm $|Y_x|_1$ satisfies equation~\eqref{e:greedyCondition} and $M_d(|Y_x|_1)<\frac12$. Then the perturbed lattice with perturbations $\{Y_x\}_{x\in\mathbb{Z}^d}$ is $k$-deletion singular for all $k\geq 1$.
\end{lemma}

\begin{proof}
We consider only the case of $k=1$, the case of  larger $k$ following essentially without change.
With $\gamma=(u_0,u_1,\ldots)$ and $\gamma_n$ defined as above, for a countable set of points $A\subset \Z^d$ define
\[
f(A) = \inf_{\psi:\Z^d \to A} \sup_{\gamma}  \limsup_n \frac1n \sum_{u\in \gamma_n} |\psi(u)-u|_1,
\]
where the supremum is over all paths $\gamma$ and the infimum is over all bijections from $\Z^d$ into $A$.  Taking $A=\Pi$ and $\psi(u) = u +Y_u$ we have that
\[
f(A) \leq \sup_{\gamma}  \limsup_n \frac1n \sum_{u\in \gamma_n} |Y_u|_1 < \frac12,
\]
since   $M_d(|Y_u|_1)<\frac12$.

Now consider $f(\Pi_0)$.  We define the bijection $W:\Pi \to \mathbb{Z}^d$ so that $y=W(y)+Y_{W(y)}$; this is uniquely
defined almost surely since the $Y_x$ have distributions with no atoms.  Given a bijection $\psi: \Z^d \to \Pi_0$, construct a path $\gamma$ as follows.
Let $v_0=0$ and set $v_{j+1}= W(\psi(v_j))$ for $j \geq 1$ and let $s_j = \sum_{k=1}^j |v_k - v_{k-1}|_1$.  Suppose that $v_j=v_{j'}$ for some $j'>j$.  Then
\[
\psi(v_{j-1})= W^{-1} (v_{j}) = W^{-1} (v_{j'}) = \psi(v_{j'-1})
\]
and so $v_{j-1}=v_{j'-1}$.  Iterating we have that $0=v_0=v_{j'-j}$ which is a contradiction since $v_{j'-j}\in W(\Pi_0)= \mathbb{Z}^d\setminus \{0\}$.

Let $\gamma$ be a nearest neighbor path constructed by sequentially joining the $v_i$ with the shortest intermediate paths, that is $\gamma=(u_0,\ldots)$ satisfies $u_0=v_0=0$ and $u_{s_j} = v_j$.  Then since $Y_{v_{k+1}} = W^{-1}(v_{k+1}) - v_{k+1} = \psi(v_k)- v_{k+1}$ we have that,
\begin{align*}
\limsup_n \frac1n \sum_{u\in \gamma_n} |\psi(u)-u|_1 &\geq \limsup_j \frac1{s_j} \sum_{k=0}^j |\psi(v_k)-v_k|_1,\\
& \geq \limsup_j \frac1{s_j} \sum_{k=0}^j |v_{k+1}-v_j|_1 - |v_{k+1}-\psi(v_k)|_1\\
&\geq \limsup_j \frac1{s_j} \left( s_j - \sum_{k=0}^j |Y_{v_{k+1}}|_1 \right)\\
&\geq 1 - \limsup_n \frac1n \sum_{u\in \gamma_n} |Y_u|_1 > \frac12.
\end{align*}

Since almost surely $f(\Pi) < \frac12$ and $f(\Pi_0)>\frac12$ we have that the two measures are mutually singular.
\end{proof}

\subsection{Proof of Theorem~\ref{t:gaussian}}

\begin{proof} 
If $g(x)=\frac1{\sqrt{2\pi}\sigma} e^{-x^2/(2 \sigma^2)}$ is a one-dimensional Gaussian $N(0,\sigma^2)$ density then
\begin{align}\label{e:GaussianRNration}
\int_{\mathbbm{R}} \left(\frac{g(x+1)}{g(x)} \right)^2 g(x) dx
& = \int_{\mathbbm{R}} \left(\frac{\exp[-(x-1)^2/2\sigma^2]}{\exp[-x^2/2\sigma^2]} \right)^2 g(x) dx\\ \nonumber
& = \int_{\mathbbm{R}} \exp[(2x-1)/\sigma^2] g(x) dx\\ \nonumber
& =  \exp[1/\sigma^2].
\end{align}
As the $d$-dimensional Gaussian measure with density $g_d(x)$ is a product measure, when calculating~\eqref{e:muIcondition} the contributions to the product not in the direction of $e_i$ cancel and the equation reduces to
\begin{equation*}
\int_{\mathbbm{R}^d} \left(\frac{g_d(x+e_i)}{g_d(x)} \right)^2 g_d(x) d x = \int_{\mathbbm{R}} \left(\frac{g(x+1)}{g(x)} \right)^2 g(x) d x  =  \exp[1/\sigma^2].
\end{equation*}
It follows from Proposition~\ref{p:walkTolerant} that for sufficiently large $\sigma$, the process $\Pi$ is deletion and insertion tolerant and mutually absolutely continuous with $\Pi_0$.

We now consider the case when $\sigma$ is small.  By scaling, the greedy lattice animal with weights $|Y_x|_1$ has a finite limiting value with $M(|Y_x|_1)$ proportional to $\sigma$.  It follows by Lemma~\ref{l:greedyIntolerant} that $\Pi$ is deletion singular for sufficiently small $\sigma>0$.

The existence of a critical value $\sigma_c(d)$ follows from the observation that increasing $\sigma$ is equivalent to a semigroup acting on $\Pi$ by shifting the points according to independent Brownian motions.  If $\Pi$ and $\Pi_0$ are not singular for some value of $\sigma$ then they can be coupled with positive probability and hence they can be coupled for all larger values of $\sigma$ as well.  Hence, by Proposition~\ref{p:equivalence}, there must be a critical $\sigma_c(d)$ with deletion tolerance for $\sigma>\sigma_c(d)$ and deletion singularity for $\sigma<\sigma_c(d)$.

We similarly have that for each $k$ there exists a threshold $\sigma_c(k,d)$ with $k$-deletion tolerance above $\sigma_c(k,d)$ and $k$-deletion singularity below.  Letting $\sigma_r(d)=\inf_k \sigma_c(k,d)$ by Proposition~\ref{p:k-equivalence} when $\sigma>\sigma_r$ there is some $k$ for which $\Pi$ is not $k$-deletion singular and hence not rigid by Proposition~\ref{p:rigid-k-tolerant}.  Conversely, if $\sigma < \sigma_r$, then $\Pi$ is $k$-deletion singular for all $k$ and hence rigid by Propositions~\ref{p:k-equivalence} and~\ref{p:rigid-k-tolerant}.  It follows from Lemmas~\ref{p:walkTolerant} and~\ref{l:greedyIntolerant} that $0<\sigma_r(d)<\infty$; this completes the proof.
\end{proof}

\section{General Perturbations}\label{s:generalPerturbation}

In this section we consider more general perturbations and analyze the effect of tails on deletion tolerance.  In particular, we exhibit a transition occurring at a power law decay of exponent $-2d$.

\begin{theorem}\label{t:general}
Let $\Pi$ be the perturbed lattice with perturbations $Y_x$ with density $g(y)$.
\begin{itemize}
\item If $\alpha< 2 d$ and
\begin{equation}\label{e:heavyTailCond}
\inf_{x\in \mathbb{R}^d} \frac{g(x)}{1\wedge |x|^{-\alpha}}  > 0 \, ,
\end{equation}
then the perturbed lattice $\Pi$ is $k$-deletion and $k$-insertion tolerant for all $k$ and mutually absolutely continuous with $\Pi_S$ for any finite set $S\subset\mathbb{Z}^d$.
\item If $\alpha > 2 d$ and
\begin{equation}\label{e:lightTailCond}
\sup_{x\in \mathbb{R}^d} \frac{g(x)}{1\wedge |x|^{-\alpha}}  < \infty \, ,
\end{equation}
then there exists $\epsilon$ such that the perturbed lattice with perturbations $\epsilon' Y_x$ is $k$-deletion singular for all $0<\epsilon'<\epsilon$ and all $k$.  This result also holds under the condition that $\E|Y_x|^{\alpha-d}<\infty$.
\end{itemize}
\end{theorem}

\begin{proof}

We first establish the second half of the theorem  when the tails are  light.  The assumption on the density in equation~\eqref{e:lightTailCond}, or the assumption $\E|Y_x|^{\alpha-d}<\infty$, both imply that equation~\eqref{e:greedyCondition} holds for $|Y_x|_1$,  so for small  enough $\epsilon>0$, the limiting constant from~\eqref{e:greedyConvergence} satisfies $M_d(|\epsilon' Y_x|_1)<\frac12$ for $\epsilon'<\epsilon$.  Applying Lemma~\ref{l:greedyIntolerant} then establishes that the perturbed lattice with perturbations $\epsilon' Y_x$ is $k$-deletion singular  completing the proof.  The remainder of the section is devoted to establishing the first half of Theorem~\ref{t:general}.

We will prove the claim in the case $k=1$, with the extension to larger $k$ following similarly.  Let $B_r(0)$ be the Euclidian ball of radius $r$ around the origin.
Define a partition of $\mathbb{R}^d$ into subsets $\{H_i\}_{i\geq 1}$ by $H_1=B_2(0)$, and  $H_i=B_{2^i}(0)\setminus H_{i-1}$ for $i\geq 2$.

Given that equation~\eqref{e:heavyTailCond} holds, the density $g$ is everywhere positive. Since the perturbations are independent, it is sufficient to show that $\Pi_0$ can be coupled with positive probability to $\hpi$, the perturbed point process identical to $\Pi$ except that the perturbation of $0$ is taken according to the uniform distribution on $H_1 \cup H_2$ instead of as $Y_0$.  We denote these perturbations as $\hat{Y}_x$ and will construct a coupling so that $\P(\hpi=\Pi_0)>0$.

By construction there exist constants $0<c_1<c_2$ such that
\begin{equation}\label{e:HiSize}
c_1 2^{id}\leq \left|H_i\cap(\mathbb{Z}^d\setminus\{0\})\right| \leq c_2 2^{id},\qquad c_1 2^{id}\leq \left|H_i\right| \leq c_2 2^{id} \, .
\end{equation}
By equation~\eqref{e:heavyTailCond} then we have that for some $c_3>0$ and for all $i\geq 1$,
\[
\inf_{x\in H_i,y\in H_i\cup H_{i+1}} g(y-x) \geq c_3 2^{-\alpha i} \, .
\]
It follows that with $c_4=c_1c_3$ that for all $i$ and $x\in H_i\cup \mathbb{Z}^d\setminus\{0\}$ we can decompose the measure of $x +Y_x$ into a mixture of the uniform distribution on $H_i\cup H_{i+1}$ with probability $p_i = c_4 2^{i(d-\alpha)}$ and another probability measure $\mu_x$ with probability $1-p_i$.

The first step of our coupling is to construct independent Bernoulli random variables $\{\zeta_x\}_{x\in\mathbb{Z}\setminus\{0\}}$ where $\P(\zeta_x=1)=p_i$ when $x\in H_i$.  When $\zeta_x=0$ we choose $x+\hat{Y}_x$ according to $\mu_x$ and let $x+\hat{Y}_x=x+Y_x$ so it remains to couple the vertices with $\zeta_x=1$ which are distributed uniformly on $H_i\cup H_{i+1}$.  Let $\mathcal{Z}= \{Z_i\}_i$ where $Z_i$ denotes the number of $x\in H_i\cap(\mathbb{Z}\setminus\{0\})$ with $\zeta_x=1$. Counting the fact that $\hat{Y}_0$ is uniform on $H_1\cup H_2$ set $\hat{Z}_1 = 1+Z_1$ and $\hat{Z}_i = Z_i$ for $i\geq 2$.  In summary the remaining not yet coupled points in $\Pi_0$ (respectively $\hpi$) are $Z_i$ (resp. $\hat{Z}_i$) points independent and uniform in $H_i\cup H_{i+1}$ for $i\geq 1$.

Now sampling according to the uniform distribution on $H_i\cup H_{i+1}$ is equivalent to first selecting $H_i$ or $H_{i+1}$ with probability proportional to their area and then sampling the selected region uniformly.  So set $r_i=\frac{|H_i|}{|H_i\cup H_{i+1}|}$ and note that the $r_i$ are uniformly bounded away from $0$ and $1$.  Hence we define as binomials, $W_i = B(Z_i,r_i)$ and set $\mathcal{U}= \{U_i\}_i$ so that $U_1=W_1$ and $U_i=W_i+Z_{i-1}-W_{i-1}$ for $i\geq 2$.  Define $\hat{W}_i$ and $\hat{\mathcal{U}}=\{\hat{W}_i\}_i$ similarly.  With these definitions  the remaining not yet coupled points in $\Pi_0$ (respectively $\hpi$) are $U_i$ (resp. $\hat{U}_i$) points independent and uniform in $H_i$ for each $i\geq 1$.

So our procedure for coupling the remaining points is as follows.  Given~$\mathcal{Z}$, we take the coupling maximizing the probability that $\mathcal{U} \equiv \hat{\mathcal{U}}$.  Conditional on this event the remaining points in $\Pi_0$ and $\hat{\Pi}$ have the same law, namely the union of $U_i$ independent uniformly chosen points in $H_i$ for each $i \geq 1$.  Thus on the event $\mathcal{U} \equiv \hat{\mathcal{U}}$ we can couple $\Pi_0$ and $\hat{\Pi}$ and hence to show deletion tolerance it remains to establish that we can couple $\mathcal{U}$ and $\hat{\mathcal{U}}$ with positive probability.
\begin{claim}
With the definitions above, $\P(\mathcal{U} \equiv \hat{\mathcal{U}})>0$.
\end{claim}
With $c_5=\tfrac12 c_1 c_4>0$ denote by $\mathcal{E}$ the event that,
\begin{equation}\label{e:ZiCond}
Z_i \geq 1\vee c_5 2^{i(2d-\alpha)}, \ \ i \geq 1 \, .
\end{equation}
We will show that $\P(\mathcal{U} \equiv \hat{\mathcal{U}}\mid \mathcal{Z},\mathcal{E})>0$.  First, we will check that $\P(\mathcal{E})>0$. By construction each $Z_i$ is independent with distribution $B(H_i\cap(\mathbb{Z}\setminus\{0\},p_i)$ and so by equation~\eqref{e:HiSize} we have that $\E Z_i \geq c_1 c_4 2^{i(2d-\alpha)}$.  Hence with our choice of $c_5$ by the Azuma-Hoeffding inequality,
\[
\P(Z_i \leq c_5 2^{i(2d-\alpha)}) \leq c_6\exp[-c_7 2^{i(2d-\alpha)}] \, ,
\]
for large $i$. Given this (better than) exponential decay and as $\P(Z_i \geq 1\vee c_5 2^{i(2d-\alpha)})>0$ for all $i$ it follows that $\P(\mathcal{E})>0$.

Now for $\mathcal{U} \equiv \hat{\mathcal{U}}$ we must have $W_1=\hat{W}_1$ and $W_i=1+\hat{W}_i$ for all $i\geq 2$.  The optimal coupling is at least as good as taking the optimal coupling independently for each $i$  so we have that
\begin{align}
\P(\mathcal{U} \equiv \hat{\mathcal{U}}\mid \mathcal{Z}) \geq (1-d_{\mathrm{TV}}(W_1,\hat{W}_1\mid \mathcal{Z}))\prod_{i=2}^\infty (1-d_{\mathrm{TV}}(W_i,\hat{W}_i+1\mid \mathcal{Z}))
\end{align}
where $d_{\mathrm{TV}}(\cdot,\cdot\mid \mathcal{Z})$ denotes the total variation distance given $\mathcal{Z}$.  Since $d_{\mathrm{TV}}(W_1,\hat{W}_1\mid \mathcal{Z})<1$ and $d_{\mathrm{TV}}(W_i,\hat{W}_i+1\mid \mathcal{Z},\mathcal{E})<1$ for all $i\geq 2$ it is sufficient to show that
\begin{align}\label{e:WTVBound}
\sum_{i=2}^\infty d_{\mathrm{TV}}(W_i,\hat{W}_i+1\mid \mathcal{Z},\mathcal{E}) < \infty.
\end{align}
Hence we estimate $d_{\mathrm{TV}}(B(n,p),B(n,p)+1)$.  When $p \leq \frac12$ and  $|j-np|\leq (np)^{3/4}$ then
\begin{align}\label{e:binomialRatio}
&\frac{\left|\P(B(n,p)=j)-\P(B(n,p)=j-1)\right|}{\P(B(n,p)=j)} \nonumber\\
&\qquad= \frac{\left|{n \choose j}p^j(1-p)^{n-j} - {n \choose j-1}p^{j-1}(1-p)^{n-j+1}\right|}{{n \choose j}p^j(1-p)^{n-j}}\nonumber\\
&\qquad= \left|1-\frac{j }{np}\cdot \frac{n(1-p)}{(n-j+1)}\right|\nonumber\\
&\qquad= \left|1-\left(1 + \frac{j -np }{np}\right)\left( 1 - \frac{j - np - 1}{n(1-p)}\right)^{-1}\right|\nonumber\\
&\qquad \leq c_8 (np)^{-1/4}
\end{align}
provided  $np$  an $n(1-p)$ are sufficiently large. It follows that
\begin{align}\label{e:binomialTV}
d_{\mathrm{TV}}(B(n,p),B(n,p)+1)
&= \frac12\sum_{j=0}^{n+1} \left|\P(B(n,p)=j)-\P(B(n,p)=j-1)\right| \nonumber\\
&\leq \P(|B(n,p)-np|\geq (np)^{3/4})\nonumber\\
&\quad + c_8(np)^{-1/4}\sum_{j=np-(np)^{3/4}}^{np+(np)^{3/4}}\P(B(n,p)=j) \nonumber\\
&\leq 2\exp\left(-\frac{(np)^{3/2}}{n}\right)+ c_8 (np)^{-1/4}
\end{align}
where the first inequality is by equation~\eqref{e:binomialRatio} and the second is by Azuma's inequality.  Since $r_i$ is uniformly bounded away from 0 and 1 then
\begin{align}\label{e:WTVBound2}
d_{\mathrm{TV}}(W_i,\hat{W}_i+1\mid \mathcal{Z}) &= d_{\mathrm{TV}}(B(Z_i,r_i),B(Z_i,r_i)+1\mid \mathcal{Z})\nonumber\\
& \leq  2\exp\left(-\frac{(Z_i r_i)^{3/2}}{Z_i}\right)+ c_8 (Z_i r_i)^{-1/4}
\end{align}
Substituting~\eqref{e:WTVBound2} into equation~\eqref{e:WTVBound} we have that
\begin{align*}\label{e:WTVBoundFinal}
&\sum_{i=2}^\infty d_{\mathrm{TV}}(W_i,\hat{W}_i+1\mid \mathcal{Z},\mathcal{E})\nonumber\\
& \ \ \leq \sum_{i=2}^\infty 2\exp\left(-(1\vee c_5 2^{i(2d-\alpha)})^{1/2} r_i^{3/2}\right)+ c_8 ((1\vee c_5 2^{i(2d-\alpha)}) r_i)^{-1/4} < \infty,
\end{align*}
which establishes that $\P(\mathcal{U} \equiv \hat{\mathcal{U}}\mid \mathcal{Z},\mathcal{E})>0$ completing the claim.

Thus the claim ensures we can couple $\mathcal{U}$ and $\hat{\mathcal{U}}$ with positive probability which completes the coupling of $\Pi_0$ and $\hpi$ and proves that  $\Pi$ and $\Pi_0$ are not mutually singular.    Then the deletion and insertion tolerance of $\Pi$  and its mutual absolute continuity with $\Pi_0$ follow by Proposition~\ref{p:equivalence}.

\end{proof}

\section{Proof of Theorem~\ref{t:stable}}

The proof of Theorem~\ref{t:stable} is essentially complete from previous results.  When $\alpha>1$ then the perturbations have finite first moments and the deletion intolerance result of~\cite{HolSoo:10}.  When $\alpha<1$ the perturbations satisfy~\eqref{e:heavyTailCond} and so deletion and insertion  tolerance  follow from Theorem~\ref{t:general}.
The sole remaining case is to show that $\Pi$ is deletion singular when $\alpha=1$ (Cauchy perturbations) which is verified as a special case in the following subsection.

\subsection{Cauchy perturbations}

\begin{lemma}
If $d=1$ and $\{Y_x\}$ are i.i.d.\ Cauchy distributed, then the perturbed lattice is $k$-deletion singular for all $k$.
\end{lemma}

\begin{proof}
Our proof follows the approach of~\cite{HolSoo:10}.  Let $S\subset\mathbb{Z}^d$ and $\Phi_{m,x} = \max\{m-|x+Y_x|,0\}$.  We define
\[
\Psi_m(\Pi) = \frac1m\int_{-m}^m (m-|z|) \Pi(dz) = \frac1m\sum_{x\in\mathbb{Z}}\Phi_{m,x}.
\]
We similarly have
\[
\Psi_m(\Pi_S) = \frac1m\int_{-m}^m (m-|z|) \Pi_S(dz) = \frac1m\sum_{x\in\mathbb{Z}\setminus S}\Phi_{m,x}.
\]
and so
\begin{equation}\label{e:PsiDiff}
\Psi_m(\Pi) - \Psi_m(\Pi_0) = \frac1m \sum_{x\in S}\Phi_{m,x}\to |S|
\end{equation}
almost surely.  We next consider the variance of $\Psi_m(\Pi)$ which is bounded as
\begin{equation}\label{e:localVarContribution}
\var(\Phi_{m,x}) \leq \E[\Phi_{m,x} - \max\{m-|x|,0\}]^2 \leq \E[|Y_x| \wedge m]^2 \leq C m.
\end{equation}
If $|x|> 2m$ then since $|\Phi_{m,x}| \leq m$ and since the density of the Cauchy decays like $c y^{-2}$,
\begin{align}\label{e:distantVarContribution}
\var(\Phi_{m,x}) &\leq \E[\Phi_{m,x}]^2 \leq m^2 \P(Y_x \in [-m-x,m-x]) \nonumber\\
&\leq C m^2 \int_{-m-x}^{m-x} y^{-2} dy \leq C' m^3 |x|^{-2}.
\end{align}
Since the $\Phi_{m,x}$ are independent (over $x$) combining equations~\eqref{e:localVarContribution} and~\eqref{e:distantVarContribution} we have that
\begin{equation}\label{e:PhiVarBound}
\var \Psi_m(\Pi) \leq \frac1{m^2}\Bigg[ \sum_{x=-2m}^{2m} C m + \sum_{|x|>2m} C' m^3 |x|^{-2} \Bigg] = O(1).
\end{equation}
Now for $m'>m$ we calculate the covariance of $\Psi_m(\Pi)$ and $\Psi_{m'}(\Pi)$as
\begin{align*}\label{e:distantVarContribution}
\Cov(\Psi_{m},\Psi_{m'}(\Pi))&= \frac1{m m'} \sum_x \Cov(\Phi_{m,x},\Phi_{m',x}) \nonumber\\
&\leq \frac1{m m'} \sum_x \sqrt{\var(\Phi_{m,x})\var(\Phi_{m',x})} \nonumber\\
&\leq \frac1{m m'}\Bigg[ \sum_{x=-2m}^{2m} \sqrt{C^2 m m'} + \sum_{2m<|x|\leq 2m'} \sqrt{C C' m^3 m' |x|^{-2}}\nonumber\\
&\qquad + \sum_{|x|>2m'} C' m^{3/2} m'^{3/2}|x|^{-2}\Bigg]\nonumber\\
&\leq (m/m')^{1/2} \left[C_1 + C_2 \log(m'/m) + +C_3\right]\nonumber\\
&\leq C_4 (m/m')^{1/2}\log(m'/m).
\end{align*}
Then if we take $m_\ell=e^{2\ell^2}$ we have that $\Cov(\Psi_{m_\ell},\Psi_{m_{\ell'}}(\Pi)) \leq O(e^{-(\ell\vee \ell')})$ when $\ell\neq \ell'$ and hence
\[
\var \left[\frac1n \sum_{\ell=1}^n \Psi_{m_\ell}(\Pi)\right] = o(1).
\]
So we have that $\frac1n \sum_{i=1}^n \Psi_{m_i}(\Pi)-\E\Psi_{m_i}(\Pi)$ converges to 0 in probability while by~\eqref{e:PsiDiff} we have that $\frac1n \sum_{i=1}^n \Psi_{m_i}(\Pi_0)-\E\Psi_{m_i}(\Pi)$ converges to $-|S|$ in probability.  It follows that $\Pi$ and $\Pi_S$ are mutually singular and so by Proposition~\ref{p:k-equivalence} $\Pi$ is $k$-deletion singular for all $k$.
\end{proof}

\section{Absolute Continuity}

In this section we prove the equivalences of deletion intolerance and insertion intolerance and deletion and insertion singularity.

\begin{proof}[Proof of Proposition~\ref{p:equivalence}]

In this section we establish Proposition~\ref{p:equivalence}.   Let $\Q$ denote the law of $\Pi$ and for a finite set $S\subset \Z^d$ we denote $\Pi_S=\{x+Y_x:x\in \Z^d \setminus S\}$ and its law as $\Q_S$.

\emph{\eqref{it:mutAbs} $\Longleftrightarrow$ \eqref{it:delSingOld}}. If $\Pi$ and $\Pi_0$ are mutually singular then clearly $\Pi$ and $\Pi_0$ are not mutually absolutely continuous.  Now assume that $\Pi$ and $\Pi_0$ are not mutually singular but that $\Q_0$ is not absolutely continuous with respect to $\Q$.  Then we can find a measurable set $A$ such that $\Q(A)=1$ and $0<\Q_0(A)<1$ and that on $A$, $\Q_0$ is absolutely continuous with respect to $\Q$ with a Radon-Nikodym  derivative given by $\kappa(a)=\frac{d \Q_0}{d \Q}$.  We will show that $\Pi_0 \in A$ is a tail event for the $\{Y_x\}$.

For some $S\subset  \Z^d\setminus\{0\}$  define the set
\[
B=B_S := \{b:\Q_0(A\mid \Pi_{S\cup\{0\}} = b) \in (0,1) \}.
\]
Suppose that $\P[\Pi_{S\cup\{0\}} \in B]>0$.  Defining the sub-probability measure
\[
\tilde{\Q}_0(E):=\P[\Pi_0\in E\cap A, \Pi_{S\cup\{0\}} \in B],
\]
we have that
\[
\tilde{\Q}_0(A):=\P[\Pi_0\in A, \Pi_{S\cup\{0\}} \in B]=\int_B \Q_0(A\mid \Pi_{S\cup\{0\}}=b) d\Q_{S\cup\{0\}}( b)>0.
\]
Since $\tilde{\Q}_0$ is dominated by $\Q_0$   it is absolutely continuous with respect to $\Q$ and so not mutually singular.
Hence there exists a coupling of $(Y_x,\Pi)$ and an identically  distributed copy $(Y_x^\star,\Pi^\star)$ such that
\[
\P[ \Pi=\Pi_0^\star, \Pi_{S\cup\{0\}}^\star \in B] > 0.
\]
Since the points $\{x+Y_x^\star:x\in S\}$ must have images in $\Pi$ when the point processes are equal and as there are only countably many choices of $\hat{S}\subset \Z^d$ with $|\hat{S}|=|S|$ we have that for some $\hat{S}$,
\[
\P[ \Pi_{\hat{S}}=\Pi_{S\cup\{0\}}^\star, \Pi_{S\cup\{0\}}^\star \in B] > 0.
\]
As the $Y_x$ have a positive density everywhere the sets $\{x+Y_x^\star:x\in S\}$ and $\{x+Y_x:x\in \hat{S}\}$ are mutually absolutely continuous and hence the distributions $\Q_0(\cdot \mid \Pi_{S\cup\{0\}} = b)$ and $\Q(\cdot \mid \Pi_{\hat{S}} = b)$ are also mutually absolutely continuous.  Then by definition of $B$ we have that for all $b\in B$, $\Q(A \mid \Pi_{\hat{S}} = b) < 1$ and hence
\begin{align*}
\Q(A^c)\geq \P[\Pi_{\hat{S}} \in B, \Pi \not \in A] & = \int_B \Q(A^c \mid \Pi_{S\cup\{0\}}=b) d\Q_{\hat{S}}(b)>0.
\end{align*}
But $\Q(A^c)=0$ so we have a contradiction and hence $\P[\P(\Pi_0 \in A\mid \Pi_{S\cup\{0\}}) \in (0,1)]=0$ for all $S$.  This implies that $\Pi_0 \in A$ is a tail event and so by the Kolmogorov zero-one law we have that $\P[\Pi_0 \in A]=1$ since $\Q_0(A)>0$.  This contradicts our assumption that $\Q_0(A)<1$ so we have that $\Q_0$ is absolutely continuous with respect to $\Q$.  That $\Q$ is absolutely continuous with respect to $\Q_0$ follows similarly so the laws are mutually absolutely continuous.

\emph{\eqref{it:delTol} $\Longleftrightarrow$ \eqref{it:delSing} $\Longleftrightarrow$ \eqref{it:delSingOld}}.  Suppose $\Q$ and $\Q_0$ are   singular.  If $Z$ is a $\Pi$ point then by an abuse of notation let $\Pi_Z$ denote $\Pi \setminus Z$.  Let $X\in\Z^d$ be the random lattice point such that $X+Y_X=Z$ so $\Pi_Z=\Pi_X$.  Since, by translation, each $\Pi_x$ is singular to $\Pi$ so is $\Pi_X$ because $X\in\Z^d$ which is countable.  Hence $\Pi_Z$ is  mutually singular to $\Pi$ and so $\Pi$ is deletion singular and hence also deletion intolerant.

Conversely, suppose that $\Q$ and $\Q_0$ are not mutually singular, so they must be mutually absolutely continuous.  Then for any set $A$, if $\P[\Pi\in A]=0$ then $\P[\Pi_x \in A]=0$, whence $\P[\Pi_Z\in A] \leq \sum_{x\in\Z^d} \P[\Pi_x \in A] =0$. Thus $\Pi_Z$ is absolutely continuous with respect to $\Pi$, so $\Pi$ is deletion tolerant and not deletion singular.

\emph{\eqref{it:insTol} $\Longleftrightarrow$  \eqref{it:insSing} $\Longleftrightarrow$ \eqref{it:delSingOld}}.
Suppose $\Q$ and $\Q_0$ are mutually singular.  Let  $V \subset \mathbb{R}^d$ be a Borel set with Lebesgue measure $\mathcal{L}(V)\in(0,\infty)$ and $U$ a random variable independent of $\{Y_x\}_{x\in\Z^d}$ uniform on $V$.  Suppose that $\Pi \cup U$ is not mutually singular with respect to $\Pi$.  Then there exists an identically  distributed copy $(U^\star, Y_x^\star,\Pi^\star)$ and a coupling such that
\[
\P[ \Pi^\star \cup U^\star = \Pi] > 0.
\]
On the event that they agree let $X$ denote the random lattice point such that $X+Y_X=U^\star$.  For some $x$ we have $\P[X=x, \Pi^\star \cup U^\star = \Pi] > 0$ and hence $\P[\Pi^\star  = \Pi_x] > 0$.  But $\Q$ and $\Q_x$ are mutually singular which is a contradiction so $\Pi \cup V$ is  mutually singular with respect to $\Pi$ and hence $\Pi$ is  insertion singular and hence insertion intolerant.

Conversely if $\Q$ and $\Q_0$ are mutually absolutely continuous then $(U,\Pi)$ is  absolutely continuous with respect to $(Y_0,\Pi_0)$ since $Y_0$ has a positive density everywhere.  It follows that $\Pi\cup U$ is absolutely continuous with respect to $Y_0 \cup \Pi_0 = \Pi$ so $\Pi$ is insertion tolerant and hence not insertion singular.

\end{proof}

\section{Rigidity}

We begin with the following Proposition relating the $k$-deletion tolerance versions and which follows with minor modification to the proof of Proposition~\ref{p:equivalence}.
\begin{proposition}\label{p:k-equivalence}
If the distribution of $Y_x$ has a density which is everywhere positive and $S\subset \mathbb{Z}^d$ is of size $k$ then the following are equivalent
\begin{enumerate}
\item The perturbed lattice is $k$-deletion tolerant.
\item The perturbed lattice is $k$-insertion tolerant.
\item The perturbed lattice is not $k$-deletion singular.
\item The perturbed lattice is not $k$-insertion singular.
\item The measures $\Pi$ and $\Pi_S$ are mutually absolutely continuous.
\item The measures $\Pi$ and $\Pi_S$ are not mutually singular.
\end{enumerate}
\end{proposition}


The proof of Proposition~\ref{p:k-equivalence} follows by the same proof as Proposition~\ref{p:equivalence} with the minor alteration of adding or removing $k$-points instead of 1.  Finally we prove Proposition~\ref{p:rigid-k-tolerant} relating rigidity and deletion tolerance.

\begin{proof}(Proof of Proposition~\ref{p:rigid-k-tolerant})
Suppose first that there exists some $S$ such that $\Pi_S$ is not singular with respect to $\Pi$ but that $\Pi$ is rigid.  Then $N(\Pi_{\hbox{out}})= | \Pi_{\hbox{in}}| \ a.s.$ but also $N((\Pi_S)_{\hbox{out}})= | (\Pi_S)_{\hbox{in}}| \ a.s.$ since $\Pi$ is mutually absolutely continuous with respect to $\Pi_S$ by Proposition~\ref{p:k-equivalence}.  However, on the event $A=\{\forall x\in S: \ x+Y_x\in B_1(0)\}$ by definition $\Pi_{\hbox{out}}=(\Pi_S)_{\hbox{out}}$ but $\Pi_{\hbox{in}}=(\Pi_S)_{\hbox{in}}+|S|$.  Since $\P[A]>0$ this is a contradiction and so $\Pi$ is not rigid.

Now suppose that $\Pi$ is not rigid and fix some ball $B$ for which it fails.  Let $\psi(\Pi_{\hbox{out}},j)=\P[|\Pi_{\hbox{in}}|=j\mid \Pi_{\hbox{out}}]$. Since $\Pi$ is not rigid it follows that
\[
\P\left[ \psi(\Pi_{\hbox{out}},\Pi_{\hbox{in}}) < 1 \right] >0
\]
and  since $\E[|\Pi_{\hbox{in}}|\mid \Pi_{\hbox{out}}] = \sum_j j \psi(\Pi_{\hbox{out}},j)$ we have that
\[
\P\left[ \sum_{j<\Pi_{\hbox{in}}} \psi(\Pi_{\hbox{out}},j) >0 \right] >0.
\]
In particular for some positive integer $k$ we have that
\[
\P\left[ \psi(\Pi_{\hbox{out}},\Pi_{\hbox{in}}-k) >0 \right] >0.
\]
Thus we can construct an independent copy $\Pi'$ of $\Pi$ such that
\[
\P[\Pi_{\hbox{out}}'=\Pi_{\hbox{out}},|\Pi'_{\hbox{in}}|+k = |\Pi_{\hbox{in}}|]>0.
\]
Now since there are a countable number of finite subsets of $\mathbb{Z}^d$ we can find sets $S,S'\subset \mathbb{Z}^d$ with $|S|=|S'|+k$
\[
\P[\Pi_{\hbox{out}}'=\Pi_{\hbox{out}},\{x+Y'_x:x\in S'\}=\Pi'_{\hbox{in}},\{x+Y_x:x\in S\}=\Pi_{\hbox{in}}]>0
\]
and so by removing these points
\begin{equation}\label{e:set-removed-coupling}
\P[\Pi'_{S'}=\Pi_{S}]>0.
\end{equation}
Let $S^*\subset S$ with $|S^*|=k$ and let $U_1,\ldots,U_{|S'|}$ be i.i.d.\ standard $d$-dimensional Gaussians.  Then since each $U_i$ is mutually absolutely continuous with respect to $x+Y_x$ for any $x$ then $\Pi_{S}\cup\{U_1,\ldots,U_{|S'|}\}$ is mutually absolutely continuous with respect to $\Pi_{S^*}$ and $\Pi_{S}'\cup\{U_1,\ldots,U_{|S'|}\}$ is mutually absolutely continuous with respect to $\Pi'$ and hence $\Pi$.  Combining this with~\eqref{e:set-removed-coupling} implies that $\Pi$ and $\Pi_{S^*}$ are not mutually singular which completes the proof.
\end{proof}

\section{Deletion singularity without rigidity}\label{s:twoButNotOne}

In this section we prove Theorem~\ref{t:twoButNotOne}.
\begin{proof}
First we show that $\hat{\Pi}$ is $2$-deletion tolerant  if $\sigma^2-\delta^2>\sigma^2_c$.  By Theorem~\ref{t:gaussian} we have that $\Pi=\{x+Y_x:x\in \mathbb{Z}^d\}$ is deletion tolerant.  We can construct $\hat{\Pi}$ from $\Pi$ by replacing each point in $z\in\Pi$ with points $z+G_{z,1}$ and $z+G_{z,2}$ for independent $N_d(0,\delta^2)$ Gaussians $G_{z,1}$ and $G_{z,2}$.  Since $\Pi$ and $\Pi_0$ are mutually absolutely continuous, by Proposition~\ref{p:equivalence} we have that $\hat{\Pi}$ and $\hat{\Pi}_0=\{x+\hat{Y}_{x,i}:(x,i)\in ({Z}^d\setminus\{0\})\times \{1,2\}\}$ are mutually absolutely continuous.  Arguing similarly to the proof of Proposition~\ref{p:equivalence} it follows  $\hat{\Pi}$ is $2$-deletion tolerant.

To prove that $\hat{\Pi}$ is not deletion tolerant we again argue by contradiction from a coupling as in the proof of Lemma~\ref{l:greedyIntolerant}.

Let $\mathcal{V}_n$ be the set of all pairs of sequences $V_n=\Big((v_0,\ldots,v_n),(v^\star_0,\ldots,v^\star_n\Big)$ taking elements in  $\mathbb{Z}^d\times\{1,2\}$ such that the elements $v_0(1),\ldots,v_n(1)$ are distinct as are $v^\star_0(1),\ldots,v^\star_n(1)$ where $v_0^\star=(0,2)$ and with $v_0^\star=(0,2)$.  Let
\[
L_n=L_n(V_n) = \sum_{i=0}^n  |v^\star_i(1) - v_{i}(1)|_1 + \sum_{i=0}^{n-1} |v_i(1) - v^\star_{i+1}(1)|_1.
\]
Note that since the $v_i(0)$ are distinct we have that $L_n \geq n$.  We define the following collection of events for constants $C_1(d)>0$ to be fixed later
\begin{itemize}
\item Let $\mathcal{I}_n(V_n)$ (respectively $\mathcal{I}^\star_n(V_n)$) be the event that
\[
\sum_{i=0}^n \sum_{j=1}^2  |\hat{Y}_{v_i(1),j}|_1 \geq \frac12 L_n, \quad \hbox{resp. } \ \sum_{i=0}^n \sum_{j=1}^2  |\hat{Y}^\star_{v^\star_i(1),j}|_1 \geq \frac12 L_n.
\]
\item Let $\mathcal{J}_n(V_n)$ 
be the event that
\begin{align*}
&\sum_{i=0}^{n-1} I(|\hat{Y}_{v_i(1),1}-\hat{Y}_{v_i(1),2}|_1 \geq C_1) \geq \frac12 n,
\end{align*}
where $I(\cdot)$ denotes the indicator.
\item Let $\mathcal{J}^\star_n(V_n)$ 
be the event that
\[
\sum_{i=0}^{n-1} I(|(v^\star_{i}(1) + \hat{Y}^\star_{v^\star_{i}(1),3-v^\star_{i}(2)})-(v^\star_{i+1}(1) + \hat{Y}^\star_{v^\star_{i+1}(1),v^\star_{i+1}(2)})|_1 \leq C_1) \geq \frac12 n.
\]
\end{itemize}
By basic large deviations estimates since $\hat{Y}_{v_i(1),j}$ are $N_d(0,\sigma^2)$ then for sufficiently large $C_2(d)$ when $L_n \geq C_2 n$,
\begin{equation}\label{e:InBound}
\P[\mathcal{I}_n(V_n)] \leq (4d)^{-L_n}.
\end{equation}
and similarly for $\mathcal{I}^\star_n(V_n)$.  Let $\mathcal{F}_i$ be the $\sigma$-algebra generated by $\{\hat{Y}^\star_{v^\star_{i}(1),1},\hat{Y}^\star_{v^\star_{i}(1),1}\}_{1\leq i' \leq i}$.  By choosing $C_1 = C_1(d,\sigma)$ to be sufficiently small we can make
\[
\P[|(v^\star_{i}(1) + \hat{Y}^\star_{v^\star_{i}(1),3-v^\star_{i}(2)})-(v^\star_{i+1}(1) + \hat{Y}^\star_{v^\star_{i+1}(1),v^\star_{i+1}(2)})|_1 \leq C_1\mid \mathcal{F}_i ]< \frac14(4d)^{-2 C_2},
\]
for all $V_n$ and $i$ since $\hat{Y}^\star_{v^\star_{i+1}(1),v^\star_{i+1}(2)}$ is distributed as $N_d(0,\sigma^2)$ and is independent of $\mathcal{F}_i$.  Hence
\begin{equation}\label{e:JnStarBound}
\P[\mathcal{J}^\star_n(V_n)] \leq {n \choose n/2}\left(\frac14(4d)^{-2 C_2}\right)^{\frac12 n} \leq (4d)^{-C_2 n},
\end{equation}
for large enough $n$.  Finally, we may choose $\delta>0$ to be sufficiently small so that
\[
\P[|\hat{Y}_{v_i(1),1}-\hat{Y}_{v_i(1),2}|_1 \geq C_1] \leq \frac14(4d)^{-C_2},
\]
since $\hat{Y}_{v_i(1),1}-\hat{Y}_{v_i(1),2}$ is distributed as $N_d(2\delta^2)$ and hence
\begin{equation}\label{e:JnBound}
\P[\mathcal{J}_n(V_n)] \leq (4d)^{-C_2 n}.
\end{equation}
Finally we note that $\{V_n\in\mathcal{V}_n:L_n(V_n) = \ell\}\leq (2d)^{L_n}$.

Now suppose that $\hat{\Pi}$ and $\hat{\Pi}_{(0,1)}=\big\{x+\hat{Y}_{x,i}:(x,i)\in (\mathbb{Z}^d\times \{1,2\})\setminus\{(0,1)\}\big\}$ are not mutually singular.  Then there exists an identically distributed copy $(\{\hat{Y}_{x,i}^\star\},\Pi^\star)$ and a coupling so that the event $\mathcal{A}=\{\Pi = \hat{\Pi}_{(0,1)}^\star\}$ has positive probability.
We define the bijections $W:\hat{\Pi} \to \mathbb{Z}^d\times\{1,2\}$ and $W^\star:\hat{\Pi}_{(0,1)}^\star\to \mathbb{Z}^d\times\{1,2\}\setminus \{(0,1)\}$ so that
\[
y=W_1(y)+\hat{Y}_{W(y)},\qquad y=W^\star_1(y)+\hat{Y}^\star_{W^\star(y)}.
\]
On $\mathcal{A}$  define the sequence $u_0^\star=(0,1)$ and
\[
u_i=W(u^\star_{i}(1)+\hat{Y}^\star_{u^\star_{i}(1),3-u^\star_{i}(2)}), \qquad u^\star_{i+1}=W^\star(u_{i}(1)+\hat{Y}_{u_{i}(1),3-u_{i}(2)}),
\]
where $u_i(1)$ denotes the first coordinate of $u_1$.  By construction the $\{u_i(1)\}_{i\geq 0}$ are distinct as are the $\{u^\star_i(1)\}_{i\geq 0}$ and $U_n=\Big((u_0,\ldots,u_n),(u^\star_0,\ldots,u^\star_n\Big)\in\mathcal{V}_n$.  Also by construction
\begin{align*}
u^\star_i(1) + \hat{Y}^\star_{u^\star_i(1),3-u^\star_i(2)} &= u_i(1) + \hat{Y}_{u_i(1),u_i(2)},\\
u_i(1) + \hat{Y}_{u_i(1),3-u_i(2)} &= u^\star_{i+1}(1) + \hat{Y}^\star_{u^\star_{i+1}(1),u^\star_{i+1}(2)},
\end{align*}
and hence by the triangle inequality we have that
\begin{align*}
L_n(V_n) &= \sum_{i=0}^n  |u^\star_i(1) - u_{i}(1)|_1 + \sum_{i=0}^{n-1} |u_i(1) - u^\star_{i+1}(1)|_1\\
&\leq \sum_{i=0}^n \sum_{j=1}^2  |\hat{Y}_{u_i(1),j}|_1 + \sum_{i=0}^n \sum_{j=1}^2  |\hat{Y}^\star_{u^\star_i(1),j}|_1
\end{align*}
and so the event $\mathcal{I}_n(U_n) \cup \mathcal{I}^\star_n(U_n)$ holds on $\mathcal{A}$.

Again by the definition of $U_n$,
\[
\hat{Y}_{v_i(1),1}-\hat{Y}_{v_i(1),2} = (v^\star_{i}(1) + \hat{Y}^\star_{v^\star_{i}(1),3-v^\star_{i}(2)})-(v^\star_{i+1}(1) + \hat{Y}^\star_{v^\star_{i+1}(1),v^\star_{i+1}(2)})
\]
and hence at least on of $\mathcal{J}_n(U_n)$ and $\mathcal{J}^\star_n(U_n)$ holds on $\mathcal{A}$.  Hence
\begin{align*}
\P[\mathcal{A}]&\leq \sum_{V_n\in\mathcal{V}_n} \P\left[\big(\mathcal{I}_n(V_n) \cup \mathcal{I}^\star_n(V_n)\big)\cap \big(\mathcal{J}_n(V_n) \cup \mathcal{J}^\star_n(V_n)\big)\right]\\
&\leq \sum_{m= n}^{C_2 n} \sum_{\substack{V_n\in\mathcal{V}_n\\L_N(V_n)=m}} \P\left[\mathcal{J}_n(V_n) \cup \mathcal{J}^\star_n(V_n)\right]\\
&\qquad +  \sum_{m=C_2 n}^{\infty} \sum_{\substack{V_n\in\mathcal{V}_n\\L_N(V_n)=m}} \P\left[\mathcal{I}_n(V_n) \cup \mathcal{I}^\star_n(V_n)\right]
\end{align*}
By equation~\eqref{e:InBound}
\[
\sum_{m=C_2 n}^{\infty} \sum_{\substack{V_n\in\mathcal{V}_n\\L_N(V_n)=m}} \P\left[\mathcal{I}_n(V_n) \cup \mathcal{I}^\star_n(V_n)\right]  \leq \sum_{m=C_2 n}^{\infty} (2d)^{m} (4d)^{-m},
\]
and by equation~\eqref{e:JnBound} and~\eqref{e:JnStarBound}
\[
\sum_{m= n}^{C_2 n} \sum_{\substack{V_n\in\mathcal{V}_n\\L_N(V_n)=m}} \P\left[\mathcal{J}_n(V_n) \cup \mathcal{J}^\star_n(V_n)\right] \leq \sum_{m= n}^{C_2 n} (2d)^{m} (4d)^{-C_2 n}
\]
and since both of these bounds tends to 0 as $n$ tends to infinity we have that $\P[\mathcal{A}]=0$.  This is a contradiction and hence $\hat{\Pi}$ is deletion singular.
\end{proof}

\noindent {\bf Acknowledgements}.  The authors would like to thank Alexander Holroyd for useful discussions.

\end{document}